\documentclass{article}

\usepackage[left = 4cm, right = 4cm]{geometry}
\usepackage{amsthm}
\usepackage{amsfonts}
\usepackage{amsmath}
\usepackage{mathtools}
\usepackage{hyperref}
\usepackage{dsfont} 
\usepackage{xcolor}
\usepackage{eucal}

\newtheorem{thr}{Theorem}[section]
\newtheorem{prop}[thr]{Proposition}
\newtheorem{lem}[thr]{Lemma}
\newtheorem{cor}[thr]{Corollary}

\theoremstyle{remark}
\newtheorem{rmk}[thr]{Remark}

\theoremstyle{definition}
\newtheorem{df}[thr]{Definition}
\newtheorem{exe}[thr]{Example}
\theoremstyle{plain}
\newtheorem*{ass*}{\textbf{\textup{(A)}}}

\newtheorem{innercustomthm}{Assumption}
\newenvironment{customthm}[1]
  {\renewcommand\theinnercustomthm{#1}\innercustomthm}
  {\endinnercustomthm}

\newcommand{\C}{\mathcal{C}}

\newcommand{\B}{\mathcal{B}}
\newcommand{\R}{\mathbb{R}}
\renewcommand{\P}{\mathcal{P}}
\renewcommand{\d}{\, \textup{d}}
\newcommand{\m}{\mathfrak{m}}

\newcommand{\supp}{\textup{supp}}
\newcommand{\esssup}{\textup{ess sup}}

\title{Some reverse inequality in optimal mass transportation}

\author{Luigi De Pascale \\
Dipartimento di Matematica e Informatica``Ulisse Dini"\\
Università degli Studi di Firenze\\ \texttt{luigi.depascale@unifi.it}
\and Igor Pinheiro\\ Dipartimento di Matematica e Informatica ``Ulisse Dini" \\ Università degli Studi di Firenze \\ \texttt{igorvinicius.pereirapinheiro@unifi.it}}

\date{\today}

\begin{document}

\maketitle

\begin{abstract}
Controlling the $\mathcal W_\infty$ Wasserstein distance by the $\mathcal W_p$ Wasserstein distance is interesting both for theorical and numerical applications. A first paper on this problem was written several years ago \cite{BouJimRaj2007PAMS}. Some year later \cite{JylRaj2016JFA} framed it in the same inequality for more general costs which increase with the distance.
In this paper, we prove this type of inequality for optimal transport problems with pointwise cost which is a decreasing function of the distance. We show, in particular, that there is a general framework that encompasses all the cases above. \end{abstract}

\section{Description of the problem and of the results}
Let $(X,d)$ be a metric space and let $\mu, \nu \in \P (X)$ be two probability measures. Let $c:X \times X \to [0, +\infty]$ be a "pointwise transportation cost" function. Denote, as usual in this domain, by 
\[{\mathcal T} (\mu, \nu):=\{T:X \to X \ : \ T_\sharp \mu=\nu\},\]
where $T$ is a Borel map and $_\sharp$ denotes the push-forward of measures. The Monge or optimal transport problem related to the cost $c$ consists of 
\[\min_{\mathcal T (\mu,\nu)} \int_X c(x,T(x)) d\mu.\]

It is, by now, well known that the Monge's problem is difficult and may have no solutions so one resorts to its Kantorovich relaxation. We refer, for these questions, to any of the many books on optimal transport, for example \cite{villani-book, villani2021topics,santambrogio2015optimal}. The Kantorovich relaxation is 
\begin{equation}\label{prob-kant}
\C (\mu,\nu)=\min_{\lambda\in \Pi (\mu,\nu)} \int_{X\times X} c(x,y) d\lambda,
\end{equation}
where $\Pi (\mu,\nu)$ is the space of probability measures  $ \lambda$ on $X\times X$ such that  $\pi^1_\sharp \lambda =\mu$ and
$\pi^2_\sharp \lambda= \nu$. There is a natural identification between $T \in {\mathcal T (\mu,\nu)}$ and $(id \times T)_\sharp \mu \in \Pi(\mu,\nu).$
Existence of at least one optimizer for \eqref{prob-kant} holds under lower semi-continuity assumption on $c$ and some mild integrability conditions.
We will mostly focus on the Kantorovich version of the problem and we will point, when of interest, to the consequence for the minimizers of the Monge version, if any. We also introduce the $sup$, or supremal, version of the problem. 
\begin{equation}\label{prob-kant-infty}
\C_\infty (\mu,\nu):=\min_{\lambda\in \Pi (\mu,\nu)} \lambda-ess\sup c(x,y),
\end{equation}
which is, by now, well known in the literature. 
The monotonicity of the integral implies that 
\begin{equation}\label{direct-ineq}
\C (\mu,\nu) \leq \C_\infty (\mu,\nu).
\end{equation}
This paper will be focused on inequalities that {\it ``reverse"} \eqref{direct-ineq} in the sense that allows to control $\C_\infty$ by $\C$. We will focus on the case of repulsive costs and, in a subsection, we will describe the literature for the attractive case and we will add some original remarks on the topic. 
Since the costs we consider are repulsive, they make sense also in the case $\mu=\nu=\rho$.
\begin{df} Let $h:[0,\infty) \to [0,\infty]$ be a non-increasing function and $\rho \in \P(X)$. We define the following transport costs,
\begin{equation}
    \C(\rho) := \inf_{\lambda \in \Pi(\rho)} \int_{X\times X} h(d(x,y)) \d \lambda(x,y).
\end{equation}
and
\begin{equation}
    \C_\infty(\rho) : = \inf_{\lambda \in \Pi(\rho)} \lambda -\esssup \, h(d(x,y)).
\end{equation}
\end{df}

\begin{rmk} The model case for the definition above is the so called Coulomb interaction: The ambient space is $\R^3$ with the Euclidean metric and $h(t) = \frac{1}{t}$. We note that, in this case, $h$ is not defined at $t=0$, but instead we have $h(0^+):=\lim_{t\to 0^+} h(t) = +\infty$. The key instance in which $h(0^+) = +\infty$ plays a role is  the finiteness of $\C(\rho)$ as noted in \cite{ButtChamDePa2018AMO} and then studied in \cite{colombo2019continuity,Bin2020CJM} (see Prop. \ref{prop2} below). Regardless, our results hold either if $h(0^+) < \infty$ or $h(0^+) = \infty$. 
\end{rmk}

The type of reverse inequality that we will study is the following: there exists a function $\omega:[0,\infty) \to \R$ such that, $\omega(t) >0$  for $t>0$ and 
\begin{equation} \label{reverse-ineq}
\omega(\C_\infty (\rho)) \leq  \C(\rho) , 
\end{equation}
where, of course, $\omega$ should be explicit. 
This formulation of the problem was first given in \cite{JylRaj2016JFA}. The function $\omega$, should allow situations as the one of the next example and this will need some study, later on, in this paper. The next example illustrate the need of some caution in the study of this problem.

\begin{exe}\label{exe: ineq not holds} Let $h(t) = t^{-1}$ and $\rho_\epsilon \in \P(\R)$ given by $\rho_\epsilon(x) = \frac{1}{2} \delta_{0} + \epsilon \delta_{1} + (\frac{1}{2}  -\epsilon) \delta_{\frac{1}{\epsilon}}$. Then, $\C(\rho_\epsilon) = 2\epsilon(\frac{3}{2} - \epsilon) \to 0 $ as $\epsilon \to 0$, but $\C_\infty(\rho_\epsilon) = 1$.
\end{exe}

\subsection{The case of Wassertein distances}
When $c(x,y)=d^p (x,y)$ the cost   $(\min_{\lambda\in \Pi (\mu,\nu)} \int_X d^p(x,y) d\lambda)^{1/p}:= {\mathcal W_p}(\mu,\nu)$ 
defines the so called $p$-Wasserstein distance between $\mu$ and $\nu$. This is, certainly, one of the most significative cases.
In \cite{BouJimRaj2007PAMS} the authors prove that, if $\Omega$ is a bounded, connected, open subset of $\R^d$ with Lipschitz boundary 
then for every $\mu \in \P_{ac} (\Omega)$, every $\nu \in \P (\overline \Omega)$ and every $p>1$,
\[({\mathcal W_\infty}(\mu,\nu))^{p+d} \leq C_{p,d} (\Omega) \|1/f\|_\infty\ {\mathcal W_p}^p(\mu,\nu),\]
where the constant $C_{p,d} (\Omega)$ only depends on $\Omega$ and $f=\frac{d \mu}{dx}$ is the Radon-Nikodym derivative of $\mu$ with respect to the Lebesgue measure. 

So, if we consider the class of $\mu$ such that $\frac{d\mu}{dx}\geq a$ for a certain positive constant $a$ the inequality become 
\begin{equation}\label{eq: BJR}\frac{a ({\mathcal W_\infty}(\mu,\nu))^{p+d}}{C_{p,d}} \leq {\mathcal W_p}(\mu,\nu)^p.\end{equation}
This last inequality is a form of \eqref{reverse-ineq} for a specific attractive cost.
We will discuss later on the different parameters in this inequality.
The paper \cite{BouJimRaj2007PAMS} contains some open problems and several other results, among which an $L^\infty$ estimate for the $p$-optimal transport map. 

Inequality \eqref{eq: BJR} played a key role in several contexts ranging from regularity theory to PDEs. For instance, in \cite{SanDav2024Arx} the authors used \eqref{eq: BJR} to improve the $\mathcal{W}_2$-convergence rate between stationary state solutions of the Porous Media Equation; in \cite{SanTos2024Arch} the authors obtained uniform estimates on the $L^\infty$ norm of the gradient of the Kantorovich potential corresponding to a JKO step; in \cite{GolOtt2020Ann} the authors developed local versions of \eqref{eq: BJR}, extended in \cite{GutMon2022Calc} to more general costs, to obtain a purely variational approach to the partial regularity of the optimal transportation map which was later generalized in \cite{OttProRie2021Ann} to more generic settings than $\R^d$ and in \cite{FriRie2025Arx} to the Coulomb cost.

The exponent $p+d$ in estimate \eqref{eq: BJR}  is related to the Lipschitz boundary of $\Omega$, in fact, if $\Omega$ is not a Lipschitz domain, one could still have for a similar estimate, but with exponent $p+s$ and $s > d$.

\begin{exe} Let $\Omega = \{(x,y) \in \R^2 \, | \, \sqrt{|x|} \leq y \leq 1\}$ and $\mu \in \P(\Omega)$ be the uniform distribution over $\Omega.$  Fix $\alpha >$ and define $\nu = \mu\vert_{B^c((0,0),\alpha)} + \mu(B((0,0),\alpha)) \delta_{(0,0)}$. One can see that, $\mathcal W_\infty(\mu,\nu) = \alpha$.  Computing $\mathcal W_p(\mu,\nu)^p$ (up to first order) gives $\mathcal W_p (\mu,\nu)^p \approx C\alpha^{p+3}$, thus,
\[\frac{\mathcal W_p(\mu,\nu)^{p}}{\mathcal W_\infty (\mu,\nu)^{p+2}} \to 0 \text{ as } \alpha \to 0,\] on the other hand, there exists $C_{p,\Omega} >0$ such that

\begin{equation}\label{domain-dimension}
\frac{\mathcal W_\infty(\mu,\nu)^{p+3}}{C_{p,\Omega}} \leq \mathcal{W}_p(\mu,\nu)^{p}.
\end{equation}
    
\end{exe}

An in depth analysis was done a few years later in \cite{JylRaj2016JFA} in which the authors proved that: 

\begin{thr}[Th. 1.1, \cite{JylRaj2016JFA} ] Let $(X,d)$ be a complete metric space, $h:[0,+\infty) \to [0,+\infty)$ a non-decreasing function with $h(t)>0$ for all $t>0$ and let $\mu \in \P(X)$. Then, there exists a non-decreasing function $\omega:[0, +\infty) \to [0, +\infty)$ with $\omega(t)>0$ for all $t>0$ such that  
\[ \omega ({\mathcal W_\infty}(\mu,\nu)) \leq  \inf_{\lambda \in \Pi (\mu,\nu)} \int h(d(x,y)) d\lambda  \ \ \ \mbox{for all}\ \nu \in \P (spt (\mu)) \]
if and only if $spt(\mu)$ is compact and connected.
Moreover, in such case one can take $\omega(t)= \frac{1}{2}m(t/17)h(t/17)$ where
\begin{equation}\label{m-attractive}
    m(t):= \inf_{x \in spt (\mu)}\mu(B(x,t)).
\end{equation}
\end{thr}
\begin{rmk}
 The function $m$ in \eqref{m-attractive} explains how the assumption or regularity on $\partial \Omega$  and the lower bound on the density of $\mu$ effect the costant and the exponents in inequality \eqref{eq: BJR} and also explains the power $3$ (which is equal to $d+1$) in inequality \eqref{domain-dimension}.
Then a strategy of study could be first to prove an inequality as the one in the theorem above and then to find significant class of measures in which such inequality is more explicit. 
\end{rmk}

\section{Main theorem}

In what follows, we consider $(X,d)$ a complete and separable metric space, i.e. a Polish space, $\rho \in \P(X)$ a probability measure on $X$ and a continuous and decreasing function $h: [0,\infty) \to (0,\infty]$. For $N\geq 1$ fixed we also consider the cost function $c: X^N \to \R
_+\cup\{+\infty \}$ given by \[c(x_1,\dots, x_N) := \sum_{1\leq i < j \leq N} h(d (x_i,x_j)).\]
And also the functionals $\C, \C_\infty: \P(X^N) \to \R_+ \cup \{+\infty\} $ given by
    \[\C(\lambda)  = \int c(x_1,\cdots, x_N) \d \lambda (x_1,\cdots, x_N) \, \, \, \, \,  \, \, \C_{\infty}(\lambda) = \lambda -\esssup \, c(x_1,\cdots,x_N).\]
$\C(\rho)$ and $\C_\infty(\rho)$ will denote, respectivelly, the infimum of each functional over all couplings $\lambda$ in the set $\Pi(\rho) = \{\lambda \in \P(X^N) \, : \, \pi^{i}_\sharp \lambda  = \rho, 1\leq i \leq N\}$.

\begin{df} Let $\rho \in \P(X)$ and $\alpha >0$. The concentration of $\rho$ at level $\alpha$ is given by the quantity
\[\kappa_{\rho}(\alpha) = \sup_{x\in X} \rho(\overline{B}(x,\alpha)).\]
The pointwise concentration of $\rho$ is given by
\[\kappa(\rho) = \sup_{x\in X}\rho(\{x\}).\]
\end{df}

\begin{rmk}\label{pointwise-density}
The pointwise concentration, $\kappa(\rho)$, first appeared in \cite{DEPASCALE_2019}  where, in the study of optimal transport with Coulomb cost and $N$ marginal constraints, it was observed that if $\kappa(\rho) <C(N)$ then the integral transport cost is finite and the optimal transport plan has support at positive distance from the diagonal. For the same reason it was used in \cite{ButtChamDePa2018AMO}. In \cite{Bin2020CJM,colombo2019continuity} it was showed that $C(N)=\frac1N$ is the sharp costant. In particular, Th. 1.1 in \cite{colombo2019continuity} and Th.1.1 in \cite{Bin2020CJM} prove that the condition $\kappa(\rho) < \frac{1}{N}$ implies the existence of a coupling $\lambda \in \Pi(\rho)$ which has support a positive distance from the set $\{x\in X^N \, | \, x_i = x_j \text{ for some } i\neq j\}$.

In the borderline case $\kappa(\rho)=\frac{1}{N}$ nothing can be said in general and one can give examples in which both $\C(\rho)=+\infty$ and $\C_\infty (\rho)=+\infty$ or both are finite of only $\C_\infty (\rho)=+\infty$.

In our case, if $h(0^+) \coloneq \lim_{t\to0^+} h(t) < + \infty$ then the condition $\kappa(\rho) < \frac{1}{N}$ is not relevant for the finiteness of $\C_\infty(\rho)$ nor $\C(\rho)$. In the interest of generality, we prefer to allow $h(0^+) = +\infty$ and introduce the following assumption.
\end{rmk}

\begin{customthm}{\textbf{\textup{(A)}}}\label{Assumption} If $h(0^+) = +\infty$, then $\kappa(\rho) < \frac{1}{N}$.
\end{customthm}


Our main goal is to prove the following Theorem.

\begin{thr}\label{thr: main_thr_N} Let $\rho \in \P(X)$, $h:[0,+\infty) \to (0,+\infty]$ be a decreasing continuous function and assume \ref{Assumption} holds. Then,
\[\C(\rho) \geq h\left(2h^{-1}\left(\frac{\C_\infty(\rho)}{N(N-1)}\right)\right) \left[ \frac{N\kappa_\rho\left( h^{-1}\left( \frac{\C_\infty(\rho)}{N(N-1)}\right)\right) -1}{N(N-1)} \right] \]
\end{thr}

\begin{rmk} The function \[t\mapsto \m(t) \coloneq \frac{N\kappa_\rho(t) - 1}{N(N-1)}\] plays the same role as the function $m$ in \eqref{m-attractive}. In \cite{JylRaj2016JFA} the authors guarantee that the function $m$ is strictly positive by considering only the measures $\mu \in \P(X)$ with compact and connected support (see Lemma 2.1 in \cite{JylRaj2016JFA}). One may note that also $\m$ is not strictly positive. In fact, if $\kappa_\rho(t) \leq 1/N$, it will not be positive. A key step is to determine the level, as a function of $\C_\infty$, that one should compute $\m$ to guarantee positiveness. This is the content of Corollary \ref{cor: positivity}. 
\end{rmk}

\begin{rmk}  
The function $\m$, as in \cite{JylRaj2016JFA}, also depends on the marginal measure $\rho$. We will show in Section \ref{sec: Classes} that this dependence is only on how $\rho$ is distributed. Namely, for a class of measures that share certain common features, it is possible to prove that there exists a positive constant $C$, independent of each specific measure $\rho$, such that $\m(h^{-1}(\frac{\C_\infty(\rho)}{N(N-1)})) \geq C >0$.
\end{rmk}

\begin{rmk}The function \[t \mapsto h(2h^{-1}(t))\]is increasing. If $h(t) = t^{-1}$ then the formula above looks
\[\C(\rho) =\int_{X^N} \sum_{1\leq i < j <N} \frac{1}{d(x_i,x_j)} \d \lambda(x_1,\cdots, x_N) \geq \
\C_\infty(\rho) \frac{N\kappa_\rho\left(\left( \frac{N(N-1)}{\C_\infty(\rho)} \right)\right) -1}{2N^2(N-1)^2}.\]
\end{rmk}

Now we develop the necessary tools to prove Theorem \ref{thr: main_thr_N}.

\begin{df}\label{costly_sets} Let $(X,d)$ be a Polish space and $B,\beta > 0$. We define the following sets
\[ \B^{B} \coloneq \{(x_1,\cdots, x_N) \, : \, c(x_1,\cdots,x_N) > B\} \text{ and }\]
\[ \B_\beta \coloneq \{(x_1,\cdots x_N) \, : \, \min_{i\neq j} d(x_i,x_j) < \beta\}.\]
\end{df}

\begin{rmk}\label{rmk: bad set inclusions}
Note that, for the cost defined as \[\sum_{1\leq i < j \leq N} h(d(x_i,x_j))\] where $h:(0,\infty) \to (0,\infty)$ is decreasing, we have $\B_\beta \subset \B^{h(\beta)}$ since
\[
    \min_{i\neq j} d(x_i,x_j) < \beta \Rightarrow \max_{i\neq j} h(d(x_i,x_j)) > h(\beta) 
    \Rightarrow \sum_{1\leq i < j \leq N} h(d(x_i,x_j)) > h(\beta).
\]

Moreover, $\B ^{\binom{N}{2}h(\beta)} \subset \B_\beta$ since

\begin{align*}\sum_{1\leq i < j \leq N} h(d(x_i,x_j)) > \binom{N}{2}h(\beta) &\Rightarrow  \max_{i < j} h(d(x_i,x_j)) > h(\beta) \\ &\Rightarrow \min_{i<j} d(x_i,x_j) < \beta. \end{align*}
\end{rmk}

Given two subsets $A,B \subset X$ and a set $\{k_1,\dots,k_L\} \subseteq \{1,\dots, N\}$ we introduce the shorthand notation \[ A^{N-L} \times_{k_1\dots k_L} B = \prod^N_{k=1} Y_k \, \, \,\text{ where } \, \, \, Y_k = \begin{cases}
    B  \, \,\text{ if } k\in \{k_1,\dots, k_L\} \\
    A \, \, \text{ otherwise }
\end{cases}\] 

\begin{lem}\label{lem:positive_measure} Let $\rho \in \P(X)$. If $\kappa_{\rho}(\alpha) > \frac{1}{N}$, then for every $\lambda \in \Pi_N(\rho)$ we have $\lambda( \B_{2\alpha}) > 0$.    
\end{lem}
\begin{proof} Assume that $\lambda (\B_{2\alpha}) = 0$ and fix a point $z\in X$.  In particular, for any $i,j \in \{1,\dots, N\}$ with $i\neq j$, we have 
\[\lambda(X^{N-2} \times_{ij} B(z,\alpha)) = 0.\]
Thus $\rho(B(z,\alpha)) = \lambda(X^{N-1}\times_i B(z,\alpha)) = \lambda (B^c(z,\alpha)^{N-1} \times_i B(z,\alpha))$ where the family $\{B^c(z,\alpha)^{N-1} \times_i B(z,\alpha)\}_{1\leq i \leq N}$ is a family of disjoint sets. Therefore,
\[N\rho(B(z,\alpha)) = \sum^N_{i=1} \lambda (B^c(z,\alpha)^{N-1}\times_i B(z,\alpha)) = \lambda\left(\bigcup^N_{i=1} B^c(z,\alpha)^{N-1} \times_i B(z,\alpha)\right) \leq 1\]
and taking the supremum over all $z \in X$ proves the contrapositive of the Lemma.
\end{proof}

\begin{thr}\label{thr: diagonal lower bound}
    Given $\lambda \in \Pi_N(\rho)$ and $\alpha>0$, then
    \begin{equation}\label{diagonal_bound}\lambda (\B_{2\alpha}) \geq \frac{N\kappa_{\rho}(\alpha) - 1}{N(N-1)}\end{equation}
\end{thr}

\begin{proof}
    For fixed $\alpha >0$ and $x\in X$ denote $B \coloneq {B}(x,\alpha)$ and define $D_B \coloneq \{(x_1,\cdots, x_N) \, : \, x_i, x_j \in B, \text{ for some } i\neq j\}.$ Clearly, $D_B \subset \B_{2\alpha}$ and we have
    \[N\rho(B) = \sum^N_{i=1} \lambda (X^{N-1} \times_i B) = \sum^N_{i=1} \lambda ((X^{N-1} \times_i B ) \cap D_B) + \sum^N_{i=1} \lambda ((X^{N-1} \times_i B ) \setminus D_B).\]

    First we study the second summation. Given $(x_1,\dots, x_N) \in (X^{N-1} \times_i B ) \setminus D_B$, then $x_i \in B$ and $x_j \not \in B$ for all $j\neq i$. Hence, for each $i$,

    \[(X^{N-1} \times_i B ) \setminus D_B = (B^c)^{N-1} \times_i B\]moreover, for $i_1 \neq i_2$, $[(B^c)^{N-1} \times_{i_1} B] \cap  [(B^c)^{N-1} \times_{i_2} B] = \emptyset$ therefore

    \[\sum^N_{i=1} \lambda ((X^{N-1} \times_i B ) \setminus D_B) = \lambda \left( \bigcup^N_{i=1} (B^c)^{N-1} \times_i B \right) \leq 1.\]

    Now, we handle the first summation. Note that $(x_1,\cdots,x_N) \in  (X^{N-1} \times_i B ) \cap D_B $ means that $x_i\in B$ and $x_j \in B$ for some $j\neq i$. We have, for $i$ fixed,

    \begin{align*}\lambda \left( (X^{N-1}\times_i B)\cap D_B \right) &=  \lambda \left ([X^{N-1}\times_i B] \cap \bigcup_{j\neq i} [X^{N-1}\times_j B]  \right) \\ &\leq \sum_{j\neq i} \lambda ( [X^{N-1}\times_i B] \cap [X^{N-1}\times_j B]) \\ &\leq (N-1)\lambda (D_B)\end{align*}
for any $j\neq i$, where in the last inequality we used $[X^{N-1}\times_i B] \cap [X^{N-1}\times_j B] \subset D_B$. Finally,
    \[\sum^N_{i=i}\lambda \left( (X^{N-1}\times_i B)\cap D_B \right) \leq \sum^N_{i=1} (N-1)\lambda (D_B) = N(N-1) \lambda (D_B).\]
Putting everything together and recalling that $D_B \subset \B_{2\alpha}$ leads to 

    \[N\rho(B) \leq N(N-1)\lambda(\B_{2\alpha}) +1 \Rightarrow \lambda(\B_{2\alpha}) \geq \frac{N\rho(B) - 1}{N(N-1)}\]recalling that $B\coloneq {B}(x,\alpha)$, taking supremum over all $x\in X$ leads to the result.
\end{proof}

\begin{rmk} It is worth mentioning that inequality \eqref{diagonal_bound} will be the key step to prove Theorem \ref{thr: main_thr_N} in the same way as Lemma 2.2. in \cite{JylRaj2016JFA}. However, they are not obtained in the same way: In \cite{JylRaj2016JFA} the authors prove that $\lambda(\{d(x,y) \geq r\}) \geq \frac{m(r)}{2}$ for each $\lambda \in \Pi(\mu,\nu)$ as long as $r < \frac{W_\infty(\mu,\nu)}{17}$, while we prove that $\lambda(\B_{2\alpha}) \geq \m(\alpha)$ for every $\lambda \in \Pi(\rho)$ and $\alpha > 0$, then we determine a specific level $\tilde \alpha$ such that $\m(\tilde \alpha)>0$ (see Corollary \ref{cor: positivity}).
\end{rmk}

\begin{rmk} Observe that substituting $N=2$ in inequality \eqref{diagonal_bound} we get $\lambda (\B_{2\alpha}) \geq \frac{2\kappa_{\rho}(\alpha) - 1}{2}$. In the following section (see Lemma \ref{lemma_ineq}) we will show that, for $N=2$, it is possible to improve the estimate to $\lambda (B_{2\alpha}) \geq 2\kappa_{\rho}(\alpha) - 1$. Therefore, it seems reasonable to expect that the bounding function can still be improved to $\frac{N\kappa_{\rho}(\alpha) - 1}{N-1}$. In fact, in the following example we show that such improved bound may be reached if  $\kappa_{\rho} > 1/N$.
\end{rmk}

\begin{exe} Let $\alpha >0$, $p \in \big(\frac{1}{N}, 1\big]$ and $(x_1,\dots,x_N) \in \R^{Nd}$ such that $d(x_i, x_j) > 2\alpha$ for every $i\neq j$. Define \[\rho(x)  = p \delta_{x_1} + \frac{1-p}{N-1} \sum^N_{i=2} \delta_{x_i}.\] We denote by $[!N]$ the set of derangements (permutations without fixed points) of $N$ elements such that if $\sigma, \varsigma \in [!N]$, then $\sigma(i)\neq \varsigma(i)$ for every $i=1,\dots,N$. 

Consider the following plan \begin{equation} \lambda := \frac{1-p}{N-1}\Bigg [\delta_{x_1} \otimes \cdots \otimes \delta_{x_N} + \sum_{\sigma \in [!N]} \delta_{x_{\sigma(1)}}  \otimes \cdots \otimes \delta_{x_{\sigma(N)}} \Bigg ] + \frac{Np-1}{N-1} \delta_{x_1} \otimes \cdots \otimes \delta_{x_1}.\end{equation}  
Clearly,
\begin{equation}\label{eq: exe 2.7}\lambda (\B_{2\alpha}) = \frac{Np-1}{N-1} = \frac{N\kappa_{\rho}(\alpha)-1}{N-1} > \frac{N\kappa_{\rho}(\alpha)-1}{N(N-1)}.\end{equation}
\end{exe}

This is indeed a consequence of $\rho$ having atoms with more than $1/N$ mass. In the following example we provide a situation in which $\kappa(\rho) \leq 1 / N$ holds and one still gets the strict inequality in \eqref{diagonal_bound}, even for the improved lower bound.

\begin{exe} Consider $\rho \in \P(\R)$ given by $\rho(x) = \frac{1}{2}\chi_{[-1,1]}$. Let $\lambda = (\text{Id},T)_{\#}\rho$ where $T(x) = x+1$ for $x\leq 0$ and $T(x) = x-1$ for $x>0$. Then, for every $\alpha >0$ we have

\[\lambda(\B_{2\alpha}) = \begin{cases} 0 \text{ if } 0\leq\alpha < \frac{1}{2} \\
1  \text{ if } \frac{1}{2 } \leq \alpha \leq 1\end{cases}\]while
$2\kappa_{\rho}(\alpha) - 1 = 2\alpha -1 < 1$ for $0\leq \alpha <1$.
\end{exe}

The estimate of Theorem \ref{thr: main_thr_N} is meaningful if $\m\left(h^{-1}\left(\frac{\C_\infty(\rho)}{N(N-1)}\right)\right)$ is strictly positive.  That is, we need to guarantee that $\kappa_{\rho}\left(h^{-1}\left(\frac{\C_\infty(\rho)}{N(N-1)}\right)\right) > \frac{1}{N}$. Then, as said in the introduction, it would be useful to provide some class of probability measures such that $\m\left(h^{-1}\left(\frac{\C_\infty(\rho)}{N(N-1)}\right)\right) > C$, for some constant $C>0$. We start with the following Proposition, which was proven in \cite{colombo2019continuity} using the duality formulation for bounded costs.

\begin{prop}[Theorem 4.1 in \cite{colombo2019continuity}]\label{prop2} Let $\rho \in \P(X)$. If $\kappa_{\rho}(\alpha) \leq \frac{1}{N}$, then there exists $\lambda \in \Pi^N(\rho)$ such that $\lambda (\B_{\alpha}) = 0.$
\end{prop}


\begin{cor}\label{cor: positivity} Let $\rho \in \P(X)$, $h:[0,+\infty) \to (0,+\infty]$ a decreasing and continuous function, then 
\[\kappa_{\rho}\Bigg( h^{-1}\bigg(\frac{\C_\infty(\rho)}{N(N-1)}\bigg)\Bigg) > \frac{1}{N}\]
\end{cor}

\begin{proof} If $\kappa_{\rho}\left( h^{-1}\big(\frac{\C_\infty(\rho)}{N(N-1)}\big)\right) \leq \frac{1}{N}$, then by Proposition \ref{prop2}, there exists $\lambda \in \Pi^N(\rho)$ such that $\lambda(\B_{r}) = 0$ for $r = h^{-1}\left( \frac{\C_\infty(\rho)}{N(N-1)} \right)$. This means that, for every $x\in \supp (\lambda)$

\[\min_{i\neq j} d(x_i,x_j) \geq  h^{-1}\bigg(\frac{\C_\infty(\rho)}{N(N-1)}\bigg)\]observing that $N(N-1) = 2\binom{N}{2}$ and $h$ non-increasing we have
\[\sum_{1\leq i<j\leq N} h(d(x_i,x_j)) \leq \binom{N}{2}\max_{i\neq j} h(d(x_i,x_j)) \leq \frac{\C_\infty(\rho)}{2} < \C_\infty(\rho)\]which leads to a contradiction as soon as one takes the essential supremum with respect to $\lambda$ on the left hand side. 
\end{proof}

\begin{rmk} Observe that if $h(0^+) = +\infty$ and $\C_\infty(\rho) = +\infty$, then $ \kappa_\rho(h^{-1}(+\infty)) = \kappa_\rho(0) = \kappa(\rho)$ and the Corollary above amounts to saying that $\kappa(\rho) > 1/N$.
\end{rmk}



Finally, we have all the ingredients for the proof of Theorem \ref{thr: main_thr_N}.

\begin{proof}[Proof of Theorem \ref{thr: main_thr_N}] Let $\lambda \in \Pi^N(\rho)$, we have

\begin{align*}
    \int_{X^N} \sum_{1\leq i < j <N} h(d(x_i,x_j)) &\d \lambda(x_1,\cdots, x_N) \\ &\geq \int_{\B_{2h^{-1}\left(\frac{\C_\infty(\rho)}{N(N-1)} \right)}} \sum_{1\leq i < j \leq N} h(d(x_i,x_j)) \d \lambda(x_1,\cdots, x_N) \\
    & \geq h\left(2h^{-1}\left(\frac{\C_\infty(\rho)}{N(N-1)}\right)\right)\lambda\left(\B_{2h^{-1}\left(\frac{\C_\infty(\rho)}{N(N-1)} \right)}\right) \\ 
    &\geq h\left(2h^{-1}\left(\frac{\C_\infty(\rho)}{N(N-1)}\right)\right) \left[ \frac{N\kappa_\rho\left( h^{-1}\left( \frac{\C_\infty(\rho)}{N(N-1)}\right)\right) -1}{N(N-1)} \right]
\end{align*}
\end{proof}

\textbf{\section{The case of two marginals}\label{2marg}}

In this section we restrict out attention to the case of two marginals and we provide simpler proofs and sharper conclusions of the results presented in the previous section. One of the key reasons for these improvements is that, for two marginals, the costly sets $\B^B$ and $\B_\beta$ in Definition \ref{costly_sets} coincide. Another reason is the, so called, Fréchet Bounds stated below which correspond to the estimate in \eqref{diagonal_bound}.

\begin{thr}[Fréchet Bounds]\label{thr: frechet bounds} Let $(X,d)$ be a polish space and $\rho \in \P(X)$. Given $\lambda \in \Pi(\rho)$, for any measurable set $A \subset X$ it holds that 
\[\max \{0, 2\rho(A) - 1\} \leq\lambda(A\times A) \leq \rho(A)\]
\end{thr}

\begin{proof}
    The right-hand side inequality is easy: $\lambda (A\times A) \leq \lambda (A\times X) \leq \rho(A)$.
    For the left-hand side inequality, if $\rho(A) \leq \frac{1}{2}$ then $2\rho(A) - 1 \leq 0$, and since $\lambda $ is a positive measure, we still get the inequality. If $\rho(A) > \frac{1}{2}$ then, by Lemma \ref{lem:positive_measure}, $\lambda (A\times A) >0$ and
    \[\lambda (A\times A) = \lambda (A\times X) - \lambda (A\times A^c) \geq \lambda (A\times X) - \lambda (X \times A^c) = \rho(A) - \rho(A^c) = 2\rho( A) -1\]
\end{proof}
\begin{thr}\label{lemma_ineq} Let $(X,d)$ be a Polish space and $\rho \in \P(X)$. Given $\lambda\in \Pi(\rho)$, for every $\alpha > 0$ the following inequality holds
    \begin{equation}\label{thr: main_ineq} \lambda (\B_{2\alpha}) \geq 2\kappa_{\rho}(\alpha) - 1\end{equation}
\end{thr}

\begin{proof}

Fix $\alpha >0$. If $\kappa_\rho(\alpha) \leq \frac{1}{2}$, then the inequality is trivially true, as  $\lambda$ is a positive measure.
Assume $\kappa_\rho(\alpha) > \frac{1}{2}$. Then, there exists $z\in X$ such that  $\rho(B(z,\alpha)) >\frac{1}{2}$ and by the proof of Lemma \ref{lem:positive_measure} we have that $\lambda (B(z,\alpha) \times B(z,\alpha)) >0$. By the Fréchet bounds it follows that 
\[\lambda(\B_{2\alpha}) \geq \lambda (B(z,\alpha)\times B(z,\alpha)) \geq 2\rho(B(z,\alpha)) -1\] and taking the supremum over all $z\in X$ we obtain the result.



\end{proof}

\begin{exe}[Strict inequality] Consider $0<\delta < \frac{1}{2}$ and $\rho \in \P(\R)$  given by $\rho = \frac{1}{3\delta}\chi_{(0,\delta)} + \frac{1}{3\delta}\chi_{(1,1+\delta)} +\frac{1}{3\delta}\chi_{(2,2+\delta)}$. For $\lambda : \frac{1}{3}\rho\vert_{(0,\delta)} \otimes \rho \vert_{(1,1+\delta)} + \frac{1}{3}\rho\vert_{(1,1+\delta)} \otimes \rho\vert_{(2,2+\delta)} + \frac{1}{3}\rho\vert_{(2,2+\delta)} \otimes \rho\vert_{(0,\delta)}$ and $2\alpha = 1+\delta$, we have

\[\lambda(\B_{1+\delta}) = \frac{2}{3} > 2\kappa\left(\rho, \frac{1+\delta}{2}\right) - 1 = \frac{1}{3}.\]

\end{exe}






The theorem below corresponds to Proposition \ref{prop2}, but we provide a different proof for the case of two marginals.

\begin{thr}\label{thr: existence plan} Let $(X,d)$ be a polish space and $\rho \in \P(X)$. If $\kappa_\rho (\alpha) \leq \frac{1}{2}$, then there exists $\lambda \in \Pi(\rho)$ such that $\lambda(\B_\alpha) = 0$.
\end{thr}

\begin{proof} Fix $\alpha > 0$ and let us introduce the cost function $c(x,y) = \chi_{\B_\alpha}(x,y)$\footnote{$\chi_{\B_\alpha}$ denotes the characteristic function of the set $\B_\alpha$ which takes the value $1$ if $x\in \B_\alpha$ and $0$ otherwise.}. Since $\B_\alpha$ is open, the cost c is lower semicontinuous and we can guarantee the existence of a minimizer for \[\inf_{\lambda \in \Pi(\rho) } \int \chi_{\B_\alpha} (x,y) \d \lambda (x,y).\] Assume, for the sake of contradiction, that $\lambda (B_\alpha) > 0$ and let $(x,y) \in \supp (\lambda ) \cap \B_\alpha$.

From $\kappa_\rho(\alpha ) \leq \frac{1}{2}$ it follows that $\lambda (B^c(y,\alpha) \times B^c(x,\alpha) ) > 0$. In fact, 
\begin{align*}
    \lambda (B^c(y,\alpha) \times B^c(x,\alpha) ) &= 1 - \lambda ([B^c(y,\alpha) \times B^c(x,\alpha) ]^c) \\ & = 1- \lambda ( [B(y,\alpha) \times X] \cup [X \times B(x,\alpha)]) \\
    & > 1- \rho(B(y,\alpha)) - \rho(B(x,\alpha)\geq 0.
\end{align*} where the strictly inequality comes from $\lambda (B(y,\alpha) \times B(x,\alpha)) > 0 $ since $(x,y) \in B(y,\alpha) \times B(x,\alpha)$.

Let $(\tilde x, \tilde y) \in [B^c(y,\alpha) \times B^c(x,\alpha)] \, \cap \, \supp (\lambda)$. By the c-cyclicall monotonicity of the support of $\lambda$ (see \cite{santambrogio2015optimal}, Theorem 1.38) it follows that

\[ 1 \leq c(x,y) + c(\tilde x, \tilde y) \leq c(x,\tilde y) + c(\tilde x, y) = 0. \]
\end{proof}

\begin{cor}\label{large_concentration} Let $(X,d)$ be a polish space, $\rho \in \P(X)$ and $h:[0,+\infty) \to (0,+\infty]$  a decreasing continuous function. Then,
    $$\kappa_\rho\left(h^{-1}\left(\frac{\C_\infty(\rho)}{2} \right)\right) > \frac{1}{2}$$
\end{cor}

\begin{proof}.
    If $\kappa\left(h^{-1}(\frac{\C_\infty}{2})\right) \leq \frac{1}{2} $ then by Theorem \ref{thr: existence plan} there exists $\lambda \in \Pi(\rho)$ such that $\lambda(\B_\alpha) = 0 $ for $\alpha = h^{-1}(\frac{\C_\infty(\rho)}{2})$. Therefore, for all $(x,y)\in \supp( \lambda )$, \[d(x,y) \geq h^{-1}\left(\frac{\C_\infty(\rho)}{2} \right)\] and, since $h$ is decreasing,
\[\lambda - \esssup \, h(d(x,y)) \leq \frac{\C_\infty(\rho)}{2} \]
which is a contradiction.
\end{proof}

\begin{thr}\label{thr: main_thr} Let $\rho \in \P(X)$, $h:[0,+\infty) \to (0,+\infty]$ be a decreasing continuous function and assume \ref{Assumption} holds. Then,
\begin{equation}\label{eqn: main_ineq}\C(\rho) \geq h\left( 2 h^{-1}\left( \frac{\C_\infty(\rho)}{2}\right)\right) \left [ 2\kappa_\rho\left( h^{-1}\left( \frac{\C_\infty(\rho)}{2}\right)\right) -1 \right ].\end{equation} 
\end{thr}

\begin{proof}
    From Corollary  \ref{large_concentration} we have that $2\kappa_\rho\left(h^{-1}(\frac{\C_\infty(\rho)}{2})\right) - 1  > 0$. Hence, using $\alpha = h^{-1}(\frac{\C_\infty(\rho)}{2})$ in Lemma \ref{lemma_ineq} we have $\lambda\left(\B_{2h^{-1}(\frac{\C_\infty(\rho)}{2})}\right) \geq 2\kappa_\rho\left(h^{-1}(\frac{\C_\infty(\rho)}{2})\right) -1$ for every $\lambda \in \Pi(\rho)$. Therefore, proceeding in the same way as in Theorem \ref{thr: main_thr_N} we have 

    \begin{align*}\int_{\R^{2d}} h(d(x,y)) \d \lambda \geq h\left(2h^{-1}\left(\frac{\C_\infty(\rho)}{2}\right)\right) \left[2\kappa_\rho\left(h^{-1}\left(\frac{\C_\infty(\rho)}{2} \right)\right)-1\right] \end{align*}
\end{proof}

\begin{rmk} In the case of $X=\R^d$, $\rho \in \P(\R^d)$ and $h(t) = \frac{1}{t}$ the inequality \eqref{eqn: main_ineq} becomes 

\begin{equation}
    \inf_{\lambda \in \Pi(\rho)} \int_{\R^{2d}} \frac{1}{|x-y|} \d \lambda (x,y) \geq \frac{\C_\infty(\rho)}{4} \Big[2\kappa_\rho \Big(\frac{2}{\C_\infty(\rho)}\Big) -1\Big]
\end{equation}

\end{rmk}

In the example below we return to Example \ref{exe: ineq not holds} to illustrate the role of the function $\m$.

\begin{exe}[Example \ref{exe: ineq not holds} BIS]\label{Example 1.3 BIS} Let $\rho_\epsilon \in \P(\R)$ given by $\rho_\epsilon(x) = \frac{1}{2} \delta_{0} + \epsilon \delta_{1} + (\frac{1}{2}  -\epsilon) \delta_{\frac{1}{\epsilon}}$. Then, $\C(\rho_\epsilon) = 2\epsilon(\frac{3}{2} - \epsilon)$, $\C_\infty(\rho_\epsilon) = 1$ and $2\kappa_{\rho_\epsilon}(\frac{4}{\C_\infty(\rho_\epsilon)}) -1= 2(\frac{1}{2} + \epsilon) -1 = 2\epsilon$. Then, for $\epsilon >0$ small enough

\[\C(\rho_\epsilon) = 2\epsilon(\frac{3}{2} - \epsilon) > \frac{2\epsilon}{8} = \frac{\C_\infty(\rho_\epsilon)}{8} \left[2\kappa \left( \frac{4}{\C_\infty(\rho_\epsilon)}\right) -1 \right]\]
\end{exe}



\section{Classes of probability measures}\label{sec: Classes}

The function  $t \mapsto \m(t) = \frac{N\kappa_\rho (h^{-1}(\frac{t}{4})) - 1}{N(N-1)}$ is decreasing and plays a fundamental role in Theorem \ref{thr: main_thr_N}. Examples \ref{exe: ineq not holds} and \ref{Example 1.3 BIS} show that it is not possible to obtain a constant $C > 0$ (independent of $\rho$) such that \begin{equation}\label{ineq: 16}C \, \C_\infty(\rho) \leq \C(\rho). \end{equation}

Throughout this section we provide two classes of probability measures in which it is possible to control $\m$ from below and then to obtain a constant $C>0$ depending only on the class, and not on the particular probability measure, which gives \eqref{ineq: 16}.

The key idea is as follows. Observe that, for any $\rho \in \P(X)$

\begin{equation}
    \C_\infty(\rho) = \min_{\lambda \in \Pi(\rho)} h \big(\sup\left\{ \alpha \, : \,  \lambda (\B_\alpha) = 0 \right\} \big)
\end{equation} thus, for any $\lambda \in \Pi(\rho)$, 

\begin{equation}\label{4. ineq}
     h^{-1}(\C_\infty(\rho)) \geq \sup \{ \alpha \, : \, \lambda (\B_\alpha) = 0\}.
\end{equation}  If, for $\tilde \alpha > 0$ fixed, $\mathfrak{R}_{\tilde \alpha} \subset \P(X)$ is a collection of measures such that, for each $\rho \in \mathfrak{R}_{\tilde \alpha}$ there exists $\lambda \in \Pi(\rho)$ such that $\lambda (\B_{\tilde \alpha}) = 0$, it follows from \eqref{4. ineq} that $h^{-1}(\C_\infty(\rho)) \geq \tilde \alpha > 0$. If one can guarantee that, for such $\tilde \alpha$ the quantity $\mathfrak{m}(\tilde \alpha)$ is strictly positive, then the inequality in Theorem \ref{thr: main_thr_N} becomes, for $\rho \in \mathfrak{R}_{\tilde \alpha}$,

\begin{equation}
    \C(\rho) \geq  C_{\tilde \alpha} \,h\left(2h^{-1}\left(\frac{\C_\infty(\rho)^{-1}}{N(N-1)}\right)\right).
\end{equation} where $C_{\tilde \alpha} = \mathfrak{m}(\tilde \alpha) $.

In this section we will provide two classes $\mathfrak{R}_{\tilde \alpha}$: The first will be a collection of unimodal isotropic probability measures on $\R^d$, and the second will be the collection of discrete measures. To avoid obscuring the meaning of the results by notational heaviness, we will restrict ourselves to the case of two marginals and to $h(s) = s^{-1}$, but the results can be naturally posed for an arbitrary number marginals and more general decreasing functions $h$.

We start with a generic condition that allows one to obtain the constant $C_{\tilde \alpha}$.

\begin{thr}\label{thr: class tail control}
    Let $\rho \in \P(X)$, 
     $h:(0,\infty) \to (0,\infty)$ be given by $h(s) = s^{-1}$. Define $r_\rho = \sup \{r >0 \,  | \, \kappa_\rho(r) \leq 1/2\}$. Let $0< \delta< \frac12$. Then there exists $C=C(\delta)>0$ (not depending on $\rho$)  such that, if 
    \begin{equation}\label{smalltaiildouble}
        \inf_{x\in \R^d} \rho(B^c(x,2r_\rho)) \leq \delta < \frac{1}{2},
        \end{equation}
    then \[   \C_\infty(\rho) \leq C\, \C(\rho).\]
\end{thr}

\begin{proof} By Theorem \ref{thr: existence plan} with $\alpha = r_\rho$ we have $\C_\infty(\rho) \leq \frac{1}{r_\rho}$, from which we derive \begin{equation}\label{eq: thr 4.1}\frac{2}{\C_\infty(\rho)}\geq 2r_\rho \end{equation}

From $\inf_{x\in \R^d} \rho(B^c(x,2r_\rho)) \leq \delta$ it follows that $\kappa_\rho(2r_\rho) \geq 1-\delta>\frac{1}{2}$. From \eqref{eq: thr 4.1} we obtain
\[2\kappa_\rho\Big(\frac{2}{\C_\infty}\Big) - 1 \geq 2\kappa_\rho(2r_\rho) - 1 > 1-2\delta,\]
which implies that $\mathfrak{m}(\frac{2}{\C_\infty(\rho)})>1-2\delta$ where $\mathfrak{m}$ is the function in Theorem \ref{thr: main_thr}. We conclude by taking $C:= \frac{1-2\delta}{4}>0$.
\end{proof}


\subsection{Unimodal isotropic probability measures}





Several classes of measures satisfy the assumptions of Theorem \ref{thr: class tail control}. We will analyze the case of unimodal isotropic probability distributions which includes the Cauchy distribution, the Gaussian distribution, the Laplace distribution and the Logistic distribution, to cite a few. We will also show that for absolutely continuous probability measures which are unimodal and isotropic, the quantity $\C_\infty$ can be fully characterized as inverse of the diameter of ball (Proposition \ref{prop: characterization}) which allows us to give a better expression of $\m$. 

\begin{lem}\label{lemma_characterization} Let $\rho \in \P(\R^d)$, then
\[\C_\infty(\rho) \geq \sup_{\alpha}\left\{ \frac{1}{2\alpha} \, : \, \kappa_\rho(\alpha) -  \frac{1}{2} > 0 \right \}\]    
\end{lem}

\begin{proof}

Fix $\alpha >0$ such that $\kappa_\rho(\alpha) > \frac{1}{2}$. Let $z\in \supp (\rho) $ be such that $\rho(\bar B(z,\alpha)) > \frac{1}{2}$. By the proof of Lemma \ref{lem:positive_measure} for every $\lambda \in \Pi(\rho)$ we have that $\lambda ( \bar B(z,\alpha)\times \bar B(z,\alpha)) > 0$, that is, there exists a pair $(x,y) \in \supp(\lambda)$ such that $|x-y| \leq  2\alpha$ which implies that \[\lambda - \esssup \frac{1}{|x-y|} \geq \frac{1}{2\alpha}, \, \, \, \, \forall \, \lambda \in \Pi(\rho)\]
    taking supremum over all $\alpha$ such that $\kappa(\alpha) > 1/2$, and then infimum over all $\lambda \in \Pi(\rho)$ leads to

    \[\C_\infty(\rho) \geq \sup \left\{ \frac{1}{2\alpha} \, : \, \kappa_\rho(\alpha) - \frac{1}{2} > 0\right\}.\]

\end{proof}





The following Theorem was already proven in \cite{ButtDePaGoGi2012PhRA} (see Appendices A and B). 
\begin{thr}\label{thr: char_rad_sym} Given $\rho\in \P_{ac}(\R^d)$, there exists $T :\R^d \to \R^d$ such that $T_{\#}\rho = \rho$ and $T$ is an optimizer in $\C(\rho)$. Moreover, if $\rho$ is radially symmetric, i.e. $\rho(x) = \rho(|x|)$ then $T(x)  = -\frac{x}{|x|} \tau(|x|)$ and $\tau$ is given by
, \begin{equation}\label{eq: char_rad_sym_tau}\tau(r) = F^{-1}(1-F(r)),\end{equation}where $F(r) = \int^r_{0} d\omega_d \rho(s)s^{d-1} \d s$.
\end{thr}

In the next result we will leverage Theorem \ref{thr: char_rad_sym} to obtain a characterization of $\C_\infty(\rho)$ in terms of the inverse of the diameter the ball. Existence of optimal maps of the form \eqref{eq: char_rad_sym_tau} have been proven only in the case $N=2$ for any dimension $d$ (see \cite{CotFriKlu2013Density}), and for any number of marginals $N$ only in dimension $1$ (see \cite{ColDePaDiMa2015Multimarginal}). Counterexamples have been provided in \cite{ColStr2016Counterexamples} for spherically symmetric probability measures in $\R^2$ and $3$ marginals.  Therefore one should not expect the strategy for this characterization to work in case $N\geq 3$ and $d\geq 2$. 
We start by recalling the definition of unimodal distribution.

\begin{df} An absolutely continuous probability measure $\rho \in \P(\R^d)$ is said to be unimodal if, for every $A>0$ the set $\{x \in X \, : \, \rho(x) > A\}$ is convex.
\end{df}

\begin{prop}\label{prop: characterization} Let $\rho \in \P_{ac}(\R^d)$ be unimodal. If $\rho$ is radially symmetric, i.e. $\rho(x) = \rho(|x|)$, then
\[\C_\infty(\rho) =\sup_r \left\{\frac{1}{2r} \, : \, \kappa_{\rho}(r) - \frac{1}{2} > 0 \right\} \]
\end{prop}

\begin{proof} Let $T$ be the optimal map in Theorem \ref{thr: char_rad_sym} and $\tau$ the respective $1$-dimensional map. We want to estimate the quantity $||\text{Id} - T||_\infty$ from below.

Let $r_\rho$ be the number such that $\kappa_{\rho}(r_\rho) = \frac{1}{2}$. This number exists thanks to the continuity and positivity of the density $\rho$. Moreover, since $\rho$ is unimodal, $r\mapsto \rho(r)$ is nonincreasing, thus $\kappa_\rho(r) = \rho(B(0,r)) = F(r)$ where $F$ is the 1 dimensional cumulative distribution function defined in Theorem \ref{thr: char_rad_sym}. In particular, $F(r_\rho) = \frac{1}{2}$.


Substituting $r_\rho$ in \eqref{eq: char_rad_sym_tau} we obtain $\tau(r_\rho) = r_\rho$. Now we estimate $||\text{Id} - T||_\infty$.

\[||\text{Id} - T||_\infty \geq |x - T(x)| = ||x| + \tau(|x|) | ,\] and choosing $r=r_\rho$ we obtain $||\text{Id} - T||_\infty\geq 2r_\rho$.

Define $\lambda_T : = (\textup{Id},T)_{\#} \rho$, thanks to Theorem \ref{lemma_characterization}

\[\frac{1}{2r_\rho} \geq \lambda_T -\esssup \frac{1}{|x-y|} \geq \C_\infty(\rho) \geq \sup_{r} \left\{ \frac{1}{2r} \, : \, \kappa_{\rho}(r) > \frac{1}{2}\right\}\]and noting that the left-most term equals the right-most term, we conclude.
\end{proof}

The following example highlights that, without unimodality, one should not expect the $\C_\infty(\rho)$ characterization, as the estimate in Lemma \ref{lemma_characterization} might be strict.

\begin{exe}
    Consider $h(t) = \frac{1}{t}$ and $\rho \in \P(\R)$  given by $\rho = \frac{1}{3\delta}\chi_{(0,\delta)} + \frac{1}{3\delta}\chi_{(1,1+\delta)} +\frac{1}{3\delta}\chi_{(2,2+\delta)}$. Then, the smallest $\alpha$ such that $\kappa(\alpha) = \frac{1}{2}$ is $\alpha = \frac{2+\delta}{4}$. On the other hand, for any $\lambda \in \Pi(\rho)$ we have $\lambda -\esssup \frac{1}{|x-y|} \geq 1$, thus
    \[\C_\infty(\rho) \geq \lambda-\esssup \frac{1}{|x-y|} \geq 1 > \frac{2}{2+\delta} = \frac{1}{2\alpha}\]
\end{exe}

An immediate consequence of the characterization of $\C_\infty(\rho)$ is an improvement on Corollary \ref{large_concentration} along with an expression for $\m$ that no longer involves $\C_\infty(\rho)$ explicitly.

\begin{cor}\label{cor: main_result} Let $\rho \in \P_{ac}(\R^d)$ be a unimodal and radially symmetric probability distribution, then
    \[\kappa_{\rho}( \C_\infty(\rho)^{-1}) = \kappa_\rho(2r_\rho) > \frac{1}{2}.\]

    Furthermore,  \[\C(\rho) \geq C_\rho \, \C_\infty(\rho) \] and $C_\rho = \kappa_\rho(2r_\rho) - \frac{1}{2} $.
\end{cor}





In the following example we highlight that $C_\rho$ depends only on how $\rho$ is distributed. We show that, for a Gaussian distribution, $C_\rho$ depends only on the dimension and not on the variance. This can be interpreted as the 68--95--99.7 rule in Statistics which states that, in a normal distribution, 68\% of the values fall within one standard deviation of the mean, 95\% within two and 99.7\% within three standard deviations. Our example highlights that, in a normal distribution, 96.87\% of the data falls within two median radius of the mean, where the median radius is the $r_\rho$ which satisfies $\kappa_\rho(r_\rho) = 1/2$.

\begin{exe}Consider the Gaussian probability measure with variance $\sigma^2$ and mean $\mu$, which we assume to be zero as it does not play a role in the computations,

\[\rho(x) = \frac{1}{(2\pi \sigma^2)^{\frac{d}{2}}} e^{-\frac{|x|^2}{2\sigma^2}}.\]

By Corollary \ref{cor: main_result}, $C_\rho = 2\kappa_\rho(2r_\rho) -1 $ depends on $\sigma$ and $d$. One can compute 

\[\kappa_{\rho}(t) = \sup_{x\in \R^d} \rho(B(x,t)) = \frac{1}{(2\pi \sigma^2)^{\frac{d}{2}}}\int_{B(0,t)} e^{-\frac{|x|^2}{2 \sigma^2}} \d x  = \frac{d\omega_d}{2\pi^{\frac{d}{2}}} \gamma\left(\frac{d}{2}, \frac{t^2}{2\sigma^2}\right),\]where $\gamma$ is the lower incomplete Gamma function. From $2\kappa_\rho(r_\rho) = 1$ we have \[ \frac{\pi^{\frac{d}{2}}{}}{d\omega_d}= \gamma\left(\frac{d}{2}, \frac{r_\rho^2}{2\sigma^2}\right)\] which shows that the function $\sigma \mapsto \gamma\left(\frac{d}{2}, \frac{r_\rho^2}{2\sigma^2}\right)$  is constant. Therefore, the function $\sigma \mapsto 2\kappa(2r_\rho) - 1$ is constant and $C_\rho$ depends only on the dimension $d$.

In the particular case of $d=2$ we can compute $\kappa_\rho$ explicitly as $\kappa_{\rho}(t) = 1-e^{-\frac{t^2}{2\sigma^2}},$ \text{ thus } \[2\kappa_{\rho}(t) -1 = 1- 2e^{-\frac{t^2}{2\sigma^2}} \text{ and } r_\rho = (2\sigma^2\ln2)^ {\frac{1}{2}}.\] Therefore, $\kappa_\rho(2r_\rho) - \frac{1}{2} = \frac{15}{32}.$







\end{exe}

We conclude this section with an application of Theorem \ref{thr: class tail control}.

\begin{prop} Let $\rho \in \P(\R^d)$ be an absolutely continuous probability measure with density of the form $\rho(x) = Ke^{-g(|x|)}$ where $K>0$ is a normalization constant and $g:[0,+\infty) \to [0,+\infty)$ is a positive and increasing convex function. For all measures of this form \[\C(\rho) \geq \frac{1}{4} \, \C_\infty(\rho).\]
\end{prop}

\begin{proof} Let $\tilde \rho \in \P([0,+\infty))$ be the probability distribution with density given by $\tilde\rho(t) =  K\exp(-g(t))t^{d-1} d \omega_d$ for $t\in [0,+\infty)$. Observe that $\tilde \rho$ is logarithmically concave, as $g$ is convex. Therefore, \[S(t) = \int^{+\infty}_{t} \tilde \rho(t) \d t \]is also logarithmically concave. 
Let $r_\rho$ be the real number which solves the equation $2\kappa_\rho(r_\rho) = 1$, then $r_\rho$ is also a solution for $2\kappa_{\tilde \rho}(r_\rho) = 1$ and we have

\begin{equation}\label{eq: 4.1.1 implication}\frac{\ln(S(2r_\rho)) - \ln(S(r_\rho))}{2r_\rho - r_\rho} \leq \frac{\ln(S(r_\rho)) - \ln(S(0))}{r_\rho - 0} \Rightarrow S(2r_\rho)) \leq \frac{1}{4}\end{equation}
where the implication comes from noting that $S(0) = 1$ and $S(r_\rho) = \frac{1}{2}$. Observing that $S(2r_\rho)) = \inf_{x\in \R^d}\rho(B^c(x,2r_\rho))$ it follows from \eqref{eq: 4.1.1 implication} and Theorem \ref{thr: class tail control} with $\delta = \frac{1}{4}$ that \[\C(\rho) \geq \frac{1}{4}\C_\infty(\rho)\]

\end{proof}


\begin{exe}
    Consider a probability measure in $\R$ distributed as \[\rho_\nu(x) = K(\nu) \frac{1}{\left(1+\frac{|x|^2}\nu\right)^{\frac{\nu+1}{2}}}\]
where $\nu \geq 1$ and $K(\nu)$ is a normalization constant. Notice that this is the radially symmetric Student-t distribution and for $\nu = 1$ one recovers the Cauchy distribution, however as $\nu \to +\infty$ one obtains the normal distribution. For $\nu = 1$ one may compute \[\rho_\nu(B^ c(0,2r_\rho)) = 1-2\frac{\arctan(2)}{\pi} > \frac{1}{4}\] while, for $\nu$ large enough, \[\rho_\nu(B^ c(0,2r_\rho)) \leq \frac{1}{4}.\]
\end{exe}


\subsection{The case of discrete measures}

\begin{thr}\label{thr: class_discrete} Let $\rho \in \P(\R^d)$ be a discrete measure, i.e.
$\rho(x) = \sum^M_{i=1} \rho_{x_i} \delta_{x_i}$ where $x_1,\dots, x_M$ are distinct points in $\R^d$ and $\rho_{x_i}$ are  positive real numbers satisfying \\ $\sum^M_{i=1} \rho_{x_i} =1$ and assume that $\rho_{x_i} < 1/2$ for every $1\leq i \leq M$. For every  $\delta > 0 $  such that \begin{equation}\label{eq: min_mass_condition}\min\limits_{1\leq i \leq M} \rho_{x_i} \geq \delta,\end{equation}then \[\C(\rho) \geq \frac{\delta}{2} \, \C_\infty(\rho).\]
\end{thr}

\begin{proof} Define $r_\rho = \sup \{r >0 \, : \, \kappa_\rho(r) \leq \frac{1}{2} \}$ and $x_\rho \in \R^d$ the point realizing the supremum of $\kappa_\rho(r_\rho)$. By Theorem \ref{thr: existence plan} we have $\C_\infty(\rho) \leq \frac{1}{r_\rho}$ and hence \[\frac{2}{\C_\infty(\rho)} \geq 2r_\rho.\]
Since $r_\rho$ is the largest radius satisfying $\kappa_\rho(r_\rho) \leq \frac{1}{2}$, it follows that $\kappa_\rho(2r_\rho) > \frac{1}{2}$. That is,

\[0 < 2\kappa(2r_\rho) - 1 = 2 \rho(\bar B(x_\rho ,2r_\rho)\setminus \bar B(x_\rho,r_\rho))\]which means that, for some $i \in \{1,\dots, M\}$ we have $x_i \in \bar B(x_\rho,2r_\rho) \setminus \bar B(x_\rho, r_\rho)$, thus \[2\kappa(2r_\rho) - 1 \geq 2\min_{1\leq i \leq M} \rho_{x_i} \geq 2\delta.\] Therefore, $\m(\frac{2}{\C_\infty(\rho)}) \geq 2\delta$ where $\m$ is the function defined in Theorem \ref{thr: main_thr} and we conclude by taking $C:= \frac{\delta}{2}$.
\end{proof}

\begin{rmk} We highlight that the Theorem \ref{thr: class_discrete} is a Corollary of \ref{thr: class tail control}. In fact, condition \eqref{smalltaiildouble} implies that $\rho(B(x_\rho,2r_\rho)\setminus B(x_\rho,r_\rho)) \geq \delta - \frac{1}{2}$ which implies a lower bound on $\{\rho_{x_i}\}^M_{i=1}$. However, we preferred to keep it as an independent result.
\end{rmk}

In \cite{AurVenAMO2022} the authors discuss estimates in the same fashion as \cite{BouJimRaj2007PAMS} and \cite{JylRaj2016JFA} for general cost functions $c:X\times X \to \R$, $X$ a Polish space, but restricted to marginals $(\mu,\nu) \in \P(X)\times \P(X)$ being discrete measures. Their strategy is as follows: Consider the set $\Pi_{\textup{opt}}(\mu,\nu)$ of couplings $\pi \in \P(X\times X)$ between $\mu$ and $\nu$ which are optimal for the cost function $c$. Among these we call \textit{trim} the one with the smallest cardinality, i.e. the \textit{trim} plan $\pi^*$ satisfies \[\# \supp(\pi^*) \leq \# \supp(\pi), \, \, \forall \, \pi \in \Pi_{\textup{opt}}(\mu,\nu).\]

For a \textit{trim} plan it is possible to find (\cite{AurVenAMO2022}, Theorem 3) a decomposition  $(\mu^{(d)},\mu^{(c)})$ of $\mu$ and a decomposition $(\nu^{(d)},\nu^{(c)})$ of $\nu$, and a pair of functions $(h^{(1)},h^{(2)})$ such that \begin{equation}
    \mu = \mu^{(d)} + \mu^{(c)} \text{ and } \nu = \nu^{(d)} + \nu^{(c)}, 
\end{equation}
\begin{equation}
    \pi^* = (\textup{Id},h^{(1)})_\sharp \mu^{(d)} + (h^{(2)},\textup{Id})_\sharp \nu^{(d)}.
\end{equation}

Such decomposition is called \textit{diffusive model} associated with the \textit{trim} minimizer $\pi^*$ and a consequence of the possibility of this decomposition is the following

\begin{thr}[\cite{AurVenAMO2022}, Corollary 4]\label{thr: AurVen} Let $\mu,\nu \in \P(X)$ be positive discrete measure and $c:X\times X$ a cost function. If $\pi^* \in \Pi_{\textup{opt}}$ is \textit{trim},  then 
\[\int_{X\times X} c(x,y) \d \pi^* \geq \alpha \min_{\pi \in \Pi(\mu,\nu)} || c(x,y)||_{L_\pi^{\infty}(X\times X)} \] where \[\alpha = \min_{x\in \supp(\mu^{(d)}), \, y \in \supp(\nu^{(d)})} \{\mu_x^{(d)}, \nu^{(d)}_y\}\]
\end{thr}

In the next proposition we show that on a Polish space $X$,  $\mu = \nu \in \P(X)$, $X$ and $c(x,y) = d(x,y)^{-1}$ our function $\m$ obtained in Theorem \ref{thr: main_thr} is bigger than $\alpha$ as defined Theorem \ref{thr: AurVen}.

\begin{prop}\label{prop: comparison_AurVen} Let $\rho \in \P(X)$ be a positive and discrete probability measure. Let $(\bar\rho^{(d)},\bar \rho^{(c)})$ and $(\underline{\rho}^{(d)}, \underline{\rho}^{(c)})$ be any diffusive system associated with a trim solution $\pi$, then \[\m\left ( \frac{2}{\C_\infty(\rho)} \right ) \geq  \alpha\]where \[\alpha = \min_{x\in \supp(\mu^{(d)}), \, y \in \supp(\nu^{(d)})} \{\mu_x^{(d)}, \nu^{(d)}_y\}\]
\end{prop}

\begin{proof} We have

\begin{equation}\label{eq: 4.2 diff model proof}\rho = \underline{\rho}^{(d)} + \underline{\rho}^{(c)} \text{ and } \rho  = \bar \rho^{(d)} + \bar \rho^{(c)}\end{equation} and $(h^{(1)}, h^{(2)})$ such that \[ \pi = (\textup{Id},h^{(1)})_\sharp \underline{\rho}^{(d)} + (h^{(2)},\textup{Id})_\sharp \bar\rho^{(d)}.\] As a consequence the decomposition \eqref{eq: 4.2 diff model proof} we have

\begin{equation}\label{eq: 4.2 min ineq proof}\min_{z\in \supp(\rho)} \rho_z \geq \min_{x \in \supp(\underline{\rho}^{(d)})}(\underline{\rho}^{(d)})_x \, \text{ and } \min_{z\in \supp(\rho)} \rho_z \geq \min_{y\in \supp(\bar \rho^{(d)})} (\bar \rho^{(d)})_y.\end{equation}

Let $r_\rho = \sup\{r > 0 \, | \, \kappa_\rho(r) \leq \frac{1}{2}\}$, then $\kappa_\rho(2r_\rho) > \frac{1}{2}$ which implies that there exists $z\in \supp(\rho)$ such that $\rho(B(x_\rho,2r_\rho) \setminus B(x_\rho, r_\rho)) \geq \rho_z$ where $x_\rho$ is the point in $X$ realizing the supremum in $\kappa_\rho(r_\rho)$. Thus,
\begin{align*}
    \m\left(\frac{2}{\C_\infty}\right) &\geq \kappa_\rho(2r_\rho ) - \frac{1}{2} \\&
    =  \rho(\bar B(x_\rho,2r_\rho) \setminus \bar B(x_\rho, r_\rho)) \\ &
    \geq \min_{z\in \supp(\rho)} \rho_{z} \\
    & \geq \min_{x\in \supp(\underline{\rho}^{(d)}), \, y \in \supp(\bar \rho^{(d)})} \{(\underline{\rho}^{(d)})_x, (\bar \rho^{(d)})_y\} = \alpha
\end{align*}where the last inequality follows from \eqref{eq: 4.2 min ineq proof}.
\end{proof}

\begin{rmk} Proposition \ref{prop: comparison_AurVen} is not claiming that constant in Theorem \ref{thr: class_discrete} is better than the constant $\alpha$ in \ref{thr: AurVen}. The number $\delta/2$ could be smaller than $\alpha$ due to the $1/2$ factor. Instead, this Proposition only highlights how a more general strategy encompasses some particular cases already present in the literature.
\end{rmk}

\bibliography{MyBiblio-LDP-2}{}
\bibliographystyle{plain}

\end{document}